\documentclass[10pt]{amsart}
\usepackage{a4}
\usepackage{amssymb}
\usepackage{array}
\usepackage{booktabs}
\usepackage{multirow}
\usepackage{hhline}
\usepackage{color}

\newtheorem{theorem}{Theorem}
\newtheorem{notation}{Notation}

\newtheorem{corollary}[theorem]{Corollary}
\newtheorem{lemma}[theorem]{Lemma}
\newtheorem{proposition}[theorem]{Proposition}
\newtheorem{definition}[theorem]{Definition}

\newtheorem{remark}[theorem]{Remark}

\numberwithin{theorem}{section}
\numberwithin{equation}{section}

\newcommand{\Comp}{\mathbb{C}}

\newcommand{\F}{\mathcal{F}}
\newcommand{\g}{\mathbb{G}}

\newcommand{\M}{\mathcal{M}}
\newcommand{\n}{\mathbb{N}}

\newcommand{\tor}{\mathbb{T}}

\newcommand{\z}{\mathbb{Z}}

\begin{document}
\title{Hardy-Littlewood inequalities on compact quantum groups of Kac type}

\author{Sang-Gyun Youn}

\address{Sang-Gyun Youn :}
\email{yun87654@snu.ac.kr}

\thanks{The author is supported by TJ Park Science Fellowship.}

\begin{abstract}
The Hardy-Littlewood inequality on $\mathbb{T}$ compares the $L^p$-norm of a function with a weighted $\ell^p$-norm of its Fourier coefficients. The approach has recently been studied for compact homogeneous spaces and we study a natural analogue in the framework of compact quantum groups. Especially, in the case of the reduced group $C^*$-algebras and free quantum groups, we establish explicit $L^p-\ell^p$ inequalities through inherent information of underlying quantum group, such as growth rate and rapid decay property. Moreover, we show sharpness of the inequalities in a large class, including $C(G)$ with compact Lie group, $C_r^*(G)$ with polynomially growing discrete group and free quantum groups $O_N^+$, $S_N^+$.
\end{abstract}
\maketitle

\section{Introduction}
Hardy and Littlewood \cite{HL27} showed that, for each $1<p\leq 2$, there exists a constant $C_p$ such that
\begin{equation}\label{HL}
(\sum_{n\in \z}\frac{1}{(1+|n|)^{2-p}}|\widehat{f}(n)|^p)^{\frac{1}{p}}\leq C_p ||f\sim \sum_{n\in \z}\widehat{f}(n)z^n||_{L^p(\tor)}~\mathrm{for~all~}f\in L^p(\tor).
\end{equation}

This implies that the multiplier $\mathcal{F}_w:L^p(\tor)\rightarrow l^p(\z)$, $f\mapsto (w(n)\widehat{f}(n))_{n\in \z}$, with $w(n):=\displaystyle \frac{1}{(1+|n|)^{\frac{2-p}{p}}}$ is bounded. Moreover, this is a stronger form of the Sobolev embedding theorem 
\[H^{\frac{1}{p}-\frac{1}{q}}_{p}(\tor)\subseteq L^{q}(\tor),~\forall 1<p<q<\infty,\]
where $H^{s}_p(\tor):=\left \{f\in L^p(\tor):(1-\Delta)^{\frac{s}{2}}(f)\in L^p(\tor)\right\}$ is the Bessel potential space.

``The Hardy-Littlewood inequality (\ref{HL})'' had been studied for compact abelian groups by Hewitt and Ross \cite{HR74}, and was recently extended to compact homogeneous manifolds by Akylzhanov, Nursultanov and Ruzhansky [\cite{ANR14} and \cite{ANR15}]. For compact Lie groups $G$ with the real dimension $n$, the result of \cite{ANR15} is written as follows: For each $1<p\leq 2$, there exists a constant $C_p>0$ such that
\begin{equation}\label{HL.cpt.gp.}
(\sum_{\pi\in \widehat{G}}\frac{1}{(1+\kappa_{\pi})^{\frac{n(2-p)}{2}}}n_{\pi}^{2-\frac{p}{2}}||\widehat{f}(\pi)||_{HS}^p)^{\frac{1}{p}}\leq C_p ||f||_{L^p(G)}\mathrm{~for~all~}f\in L^p(G).
\end{equation}

Here, $\widehat{G}$ denotes a maximal family of mutually inequivalent irreducible unitary representations of $G$, $||A||_{HS}:=\mathrm{tr}(A^*A)^{\frac{1}{2}}$ and the Laplacian operator $\Delta$ on $G$ satisfies $\Delta:\pi_{i,j}\mapsto -\kappa_{\pi}\pi_{i,j}$ for all $\pi=(\pi_{i,j})_{1\leq i,j\leq n_{\pi}}\in \widehat{G}$ and all $1\leq i,j\leq n_{\pi}$.

The above inequality (\ref{HL.cpt.gp.}) can be reduced to a more familiar form whose left hand side is a natural weighted $\ell^p$-norm of its Fourier coefficients:
	\begin{align}\label{HL-cpt.gp.2}
	(\sum_{\pi\in \widehat{G}}\frac{1}{(1+\kappa_{\pi})^{\frac{n(2-p)}{2}}} n_{\pi}||\widehat{f}(\pi)||_{S^p_{n_{\pi}}}^p)^{\frac{1}{p}} &\leq C_p ||f||_{L^p(G)}.
	\end{align}

A notable point is that ``the Hardy-Littlewood inequalities on compact Lie groups (\ref{HL.cpt.gp.})'' are determined by inherent geometric information, namely the real dimension and the natural length function on $\widehat{G}$. Indeed, $\pi\mapsto \sqrt{\kappa_{\pi}}$ is equivalent to the natural length $||\cdot||$ on $\widehat{G}$ (See Remark \ref{rmk1}).

The main purpose of this paper is to establish new Hardy-Littlewood inequalities on compact quantum groups of Kac type through geometric information of underlying quantum groups. As a part of efforts, we will present some explicit inequalities in important examples. The reduced group $C^*$-algebras $C_r^*(G)$ with discrete groups $G$, the free orthogonal quantum groups $O_N^+$ and the free permutation quantum groups $S_N^+$ are the main targets. Of course, non-commutative $L^p$ analysis on quantum groups is widely studied from various perspectives (\cite{Ca13}, \cite{FHLU15}, \cite{JMP14}, \cite{JPPP13} and \cite{Wa16}). For details of operator algebraic approach to quantum group itself, see \cite{KV00}, \cite{KV03}, \cite{Ti08} and \cite{Wo87}.

To clarify our intention, let us show the main results of this article on ``compact matrix quantum groups'' which admit the natural length function $|\cdot|:\mathrm{Irr}(\g)\rightarrow \left \{0\right\}\cup \n$ (See Definition \ref{CMQG1} and Proposition \ref{CMQG2}). The following inequalities are determined by inherent information of underlying quantum group, namely growth rate and rapid decay property.

\begin{theorem}\label{main}

Let $\g$ be a compact matrix quantum group of Kac type and denote by $|\cdot |$ the natural length function on $\mathrm{Irr}(\g)$.
\begin{enumerate}
\item Let $\g$ have a polynomial growth with $\displaystyle \sum_{\alpha\in \mathrm{Irr}(\g):|\alpha|\leq k}n_{\alpha}^2\leq (1+k)^{\gamma}$ and $\gamma>0$. Then, for each $1<p\leq 2$, there exists a universal constant $K=K(p)$ such that
\begin{align}
&(\sum_{\alpha\in \mathrm{Irr}(\g)}\frac{1}{(1+|\alpha|)^{(2-p)\gamma}}n_{\alpha}||\widehat{f}(\alpha)||_{S^p_{n_{\alpha}}}^p)^{\frac{1}{p}}\\
\notag \leq& (\sum_{\alpha\in \mathrm{Irr}(\g)} \frac{1}{(1+|\alpha|)^{(2-p)\gamma}}n_{\alpha}^{2-\frac{p}{2}}||\widehat{f}(\alpha)||_{HS}^p)^{\frac{1}{p}}\leq  K||f||_{L^p(\g)}
\end{align}
for all $f\sim \displaystyle \sum_{\alpha\in \mathrm{Irr}(\g)}n_{\alpha}\mathrm{tr}(\widehat{f}(\alpha)u^{\alpha})\in L^p(\g)$.
\item Let $\g$ have the rapid decay property with universal constants $C,\beta>0$ such that
\[||f||_{L^{\infty}(\g)}\leq C(1+k)^{\beta}||f||_{L^2(\g)}\]
for all $f\in \mathrm{span}(\left \{u^{\alpha}_{i,j}:|\alpha|=k,~1\leq i,j\leq n_{\alpha}\right\})$. Define $\displaystyle s_k:=\sum_{\alpha\in \mathrm{Irr}(\g):|\alpha|=k}n_{\alpha}^2$. Then, for each $1<p\leq 2$, there exists a universal constant $K=K(p)$ such that
\begin{align}\label{main-1-p>1}
&(\sum_{k\geq 0}\sum_{\alpha\in \mathrm{Irr}(\g):|\alpha|=k} \frac{1}{s_k^{\frac{(2-p)}{2}}(1+k)^{(2-p)(\beta+1)}}n_{\alpha}||\widehat{f}(\alpha)||_{S^p_{n_{\alpha}}}^p)^{\frac{1}{p}}\\
\notag\leq & (\sum_{k\geq 0}\frac{1}{(1+k)^{(2-p)(\beta+1)}}(\sum_{\alpha\in \mathrm{Irr}(\g):|\alpha|=k}n_{\alpha}||\widehat{f}(\alpha)||_{HS}^2)^{\frac{p}{2}})^{\frac{1}{p}}\leq K ||f||_{L^p(\g)}
\end{align}
for all $f\displaystyle \sim \sum_{\alpha\in \mathrm{Irr}(\g)}n_{\alpha}\mathrm{tr}(\widehat{f}(\alpha)u^{\alpha})\in L^p(\g)$.
\end{enumerate}
\end{theorem}

In particular, it was known that the rapid decay property of $\mathbb{F}_N$ can be strengthened in general holomorphic setting \cite{KS07}. The improved result is called the strong Haagerup inequality. Based on this data, we can improve the Hardy-Littlewood inequality for $C_r^*(\mathbb{F}_N)$ by focusing our attention on holomorphic forms. Theorem \ref{strong1} justifies the claim and it seems appropriate to call the improved one ``strong Hardy-Littlewood inequality''.

A natural view of the Hardy-Littlewood inequalities on compact Lie groups is that a properly chosen weight function $w:\widehat{G}\rightarrow (0,\infty)$ makes the corresponding multiplier 
\[\mathcal{F}_w:L^p(G)\rightarrow \ell^p(\widehat{G})\mathrm{~given~by}~f\mapsto (w(\pi)\widehat{f}(\pi))_{\pi\in \widehat{G}}\] 
be bounded for each $1<p\leq 2$. Indeed, newly derived Hardy-Littlewood inequalities on compact quantum groups will give a decay pair $(r,s)$ whose corresponding multiplier $\mathcal{F}_{w_{r,s}}$ is bounded, where $w_{r,s}(\alpha):=\displaystyle \frac{1}{r^{|\alpha|}(1+|\alpha|)^s}$. Moreover, in Section 6, we will show that there is no slower decay pair such that $\mathcal{F}_{r,s}$ is bounded when $\g$ is one of the followings: $C(G)$ with compact Lie groups, $C_r^*(G)$ with polynomially growing discrete group and free quantum groups $O_N^+$, $S_N^+$. See Theorem \ref{comp.char.}.

This approach is quite natural because it is strongly related to Sobolev embedding properties. Indeed, for the case $G=\tor^d$, $\mathcal{F}_{w_{0,s}}:L^p(\tor^d)\rightarrow l^p(\z^d)$ is bounded if and only if  
\[H^{\frac{ps}{2-p}(\frac{1}{q}-\frac{1}{r})}_q(\tor^d) \subseteq L^r(\tor^d)~\forall 1<q<r<\infty,\]
where $H^s_p(\tor^d)$ is the Bessel potential space.

Lastly, in Section 7, we present some remarks that are by-products of this research. We show that most of free quantum groups do not admit an infinite (central) sidon set, and also give a Sobolev embedding theorem type interpretation of our results for $C(G)$ with compact Lie group and $C_r^*(G)$ with polynomially growing discrete group $G$. Also, we present an explicit inequality on quantum torus $\tor^d_{\theta}$.

\section{Preliminaries}

\subsection{Compact quantum groups}

A compact quantum group $\g$ is given by a unital $C^*$-algebra $A$ and a unital $*$-homomorphism $\Delta:A\rightarrow A\otimes_{min} A$ satisfying 

$(1)~(\Delta\otimes id)\circ \Delta=(id\otimes \Delta)\circ \Delta;$

$(2)~span\left \{ \Delta(a)(b\otimes 1_A):~a,b\in A\right\},span\left \{ \Delta(a)(1_A\otimes b):~a,b\in A\right\}\mathrm{~are~dense~in~}A.$\\

Every compact quantum group admits the unique $Haar~state$ $h$ on $A$ such that
\[(h\otimes id)(\Delta(x))=h(x)1_A=(id\otimes h)(\Delta(x))~\mathrm{for~all~}x\in A.\]

A finite dimensional corepresentation of $\g$ is given by an element $\displaystyle u=(u_{i,j})_{1\leq i,j\leq n}\in M_n(A)$ such that $\Delta(u_{i,j})=\displaystyle \sum_{k=1}^n u_{i,k}\otimes u_{k,j}$ for all $1\leq i,j\leq n$. We say that the corepresentation $u$ is unitary if $u^*u=uu^*=Id_n\otimes 1_A\in M_n(A)$ and irreducible if $\left \{X\in M_n: Xu=uX\right\}=\Comp\cdot Id_n$ where $Id_n$ is the identity matrix in $M_n$.

Let $\left \{u^{\alpha}=(u^{\alpha}_{i,j})_{1\leq i,j\leq n_{\alpha}}\right\}_{\alpha\in \mathrm{Irr}(\g)}$ be a maximal family of mutually inequivalent finite dimensional unitary irreducible corepresentations of $\g$. It is well known that, for each $\alpha\in \mathrm{Irr}(\g)$, there is a unique positive invertible matrix $Q_{\alpha}\in M_{n_{\alpha}}$ such that $\mathrm{tr}(Q_{\alpha})=\mathrm{tr}(Q_{\alpha}^{-1})$ and 
\[h((u^{\beta}_{s,t})^*u^{\alpha}_{i,j})=\frac{\delta_{\alpha,\beta}\delta_{j,t}(Q_{\alpha}^{-1})_{i,s}}{\mathrm{tr}(Q_{\alpha})},~\forall \alpha,\beta\in \mathrm{Irr}(\g),1\leq i,j\leq n_{\alpha},1\leq s,t\leq n_{\beta},\]
\[~h(u^{\beta}_{s,t}(u^{\alpha}_{i,j})^*)=\frac{\delta_{\alpha,\beta}\delta_{i,s}(Q_{\alpha})_{j,t}}{\mathrm{tr}(Q_{\alpha})},~\forall \alpha,\beta\in \mathrm{Irr}(\g),1\leq i,j\leq n_{\alpha},1\leq s,t\leq n_{\beta}.\]

We say that $\g$ is of Kac type if $Q_{\alpha}=Id_{n_{\alpha}}\in M_{n_{\alpha}}$ for all $\alpha\in \mathrm{Irr}(\g)$. In this case, the Haar state $h$ is tracial.

Lastly, we define $C_r(\g)$ as the image of $A$ in the GNS representation of the Haar state $h$ and $L^{\infty}(\g):=C_r(\g)''$. The Haar state $h$ has a normal faithful extension to $L^{\infty}(\g)$.

\subsection{Non-commutative $L^p$-spaces}

Let $\M$ be a von Neumann algebra with a normal faithful tracial state $\varphi$. Note that the von Neumann algebra $\M$ admits the unique predual $\M_*$. We define $L^1(M,\varphi):=\M_*$ and $L^{\infty}(M,\varphi):=\M$, then consider a contractive injection $j:\M \rightarrow \M_*$, given by $[j(x)](y):= h(yx)$ for all $y\in \M$. The map $j$ has dense range.

Now $(\M,\M_*)$ is a compatible pair of Banach spaces and we can define non-commutative $L^p$-space $L^p(\M,\varphi):=(\M,\M_*)_{\frac{1}{p}}$ for all $1<p<\infty$, where $(\cdot, \cdot)_{\frac{1}{p}}$ is the complex interpolation space. For any $x\in L^{\infty}(\M,\varphi)$, its $L^p$-norm is given by $\displaystyle ||x||_{L^p(\M,\varphi)}=\varphi(|x|^p)^{\frac{1}{p}}$ for all $1\leq p<\infty$.

In particular, for all $1\leq p\leq \infty$, we denote by $L^p(\g)$ the non-commutative $L^p$-space associated to the von Neumann algebra $L^{\infty}(\g)$ of a Kac type compact quantum group $\g$ and the tracial Haar state $h$. Then the space of polynomials 
\[Pol(\g):=\mathrm{span}(\left \{u^{\alpha}_{i,j}:\alpha\in \mathrm{Irr}(\g),1\leq i,j\leq n_{\alpha}\right\})\]
is dense in $C_r(\g)$ and $L^p(\g)$ for all $1\leq p<\infty$.

Under the assumption that $\g$ is of Kac type, for $1\leq p <\infty$, 
\[\ell^p(\widehat{\g}):=\left \{(A_{\alpha})_{\alpha\in \mathrm{Irr}(\g)}\in \prod_{\alpha\in \mathrm{Irr}(\g)}M_{n_{\alpha}}: \sum_{\alpha\in \mathrm{Irr}(\g)}n_{\alpha}\mathrm{tr}(|A_{\alpha}|^p)<\infty \right\}\]
and the natural $\ell^p$-norm of $(A_{\alpha})_{\alpha\in \mathrm{Irr}(\g)}\in \ell^p(\widehat{\g})$ is defined by 
\[||(A_{\alpha})_{\alpha\in \mathrm{Irr}(\g)}||_{\ell^p(\widehat{\g})}:=(\sum_{\alpha\in \mathrm{Irr}(\g)}n_{\alpha}\mathrm{tr}(|A_{\alpha}|^p))^{\frac{1}{p}}=(\sum_{\alpha\in \mathrm{Irr}(\g)}n_{\alpha}||A_{\alpha}||_{S^p_{n_{\alpha}}}^p)^{\frac{1}{p}}.\]

Also, 
\[\ell^{\infty}(\widehat{\g}):=\left \{(A_{\alpha})_{\alpha\in \mathrm{Irr}(\g)}\in \prod_{\alpha\in \mathrm{Irr}(\g)}M_{n_{\alpha}}: \sup_{\alpha\in \mathrm{Irr}(\g)}||A_{\alpha}||<\infty \right\}\]
and the $\ell^{\infty}$-norm of $(A_{\alpha})_{\alpha\in \mathrm{Irr}(\g)}\in \ell^{\infty}(\widehat{\g})$ is defined by 
\[||(A_{\alpha})_{\alpha\in \mathrm{Irr}(\g)}||_{\ell^{\infty}(\widehat{\g})}:=\sup_{\alpha\in \mathrm{Irr}(\g)}||A_{\alpha}||.\]

It is known that $\ell^1(\widehat{\g})=(\ell^{\infty}(\widehat{\g}))_*$ and $\ell^p(\widehat{\g})=(\ell^{\infty}(\widehat{\g}),\ell^1(\widehat{\g}))_{\frac{1}{p}}$ for all $1<p<\infty$.

\subsection{Fourier analysis on compact quantum groups}

For a compact quantum group $\g$, the Fourier transform $\F:L^1(\g)\rightarrow \ell^{\infty}(\widehat{\g})$, $\psi\mapsto \widehat{\psi}$, is defined by
\[(\widehat{\psi}(\alpha))_{i,j}:=\psi((u^{\alpha}_{j,i})^*)~\mathrm{for~all~}\alpha\in \mathrm{Irr}(\g),1\leq i,j\leq n_{\alpha}.\]

It is also known that $\F$ is an injective contractive map and it is an isometry from $L^2(\g)$ onto $l^2(\widehat{\g})$ (\cite{Wa16}, Proposition 3.1 and 3.2). Then, by the interpolation theorem, we induce the Hausdorff-Young inequality again, i.e. $\F$ is a contractive map from $L^p(\g)$ into $l^{p'}(\widehat{\g})$ for each $1\leq p\leq 2$, where $p'$ is the conjugate of $p$.

\subsection{The reduced group $C^*$-algebras}

The reduced group $C^*$-algebra $C_r^*(G)$, can be defined for all locally compact groups, but we only consider discrete groups in this paper since we want to understand it as a compact quantum group.

\begin{definition}

Let $G$ be a discrete group and define $\lambda_g\in B(l^2(G))$ for each $g\in G$ by 
\[[(\lambda_g)(f)](x):=f(g^{-1}x)~\mathrm{for~all~} x\in G.\]
Then the reduced group $C^*$-algebra, $C_r^*(G)$, is defined as the norm-closure of the space $\mathrm{span}(\left \{ \lambda_g:~g\in G\right\})$ in $B(l^2(G))$. Moreover, if we define a comultiplication $\Delta:C_r^*(G)\rightarrow C_r^*(G)\otimes_{min}C_r^*(G)$ by $\lambda_g\mapsto \lambda_g\otimes \lambda_g$ for all $g\in G$, then $(C_r^*(G),\Delta)$ forms a compact quantum group.
\end{definition}

Note that, for $\g=(C_r^*(G),\Delta)$ of a discrete group $G$, $L^{\infty}(\g)$ is nothing but the group von Neumann algebra $VN(G)$ and $\mathrm{Irr}(\g)=\left \{\lambda_g\right\}_{g\in G}$ is identified with $G$.

\subsection{Free quantum groups of Kac type}

\begin{definition}(Free orthogonal quantum group \cite{Wa95}) 

Let $N\geq 2$ and $A$ be the universal unital $C^*$-algebra which is generated by the $N^2$ self-adjoint elements $u_{i,j}$ with $1\leq i,j\leq N$ satisfying the relations:
\[\sum_{k=1}^N u_{k,i}u_{k,j}=\sum_{k=1}^N u_{i,k}u_{j,k}=\delta_{i,j}~\mathrm{for~all~}1\leq i,j\leq N.\]

Also, we define a comultiplication $\Delta:A\rightarrow A\otimes_{min}A$ by $\displaystyle u_{i,j}\mapsto \sum_{k=1}^N u_{i,k}\otimes u_{k,j}$. Then $(A,\Delta)$ forms a compact quantum group called the Free orthogonal quantum group. We denote it by $O_N^+$.
\end{definition}

\begin{definition}(Free permutation quantum group \cite{Wa98})

Let $N\geq 2$ and $A$ be the universal unital $C^*$-algebra which is generated by the $N^2$ self-adjoint elements $u_{i,j}$ with $1\leq i,j\leq N$ satisfying the relations:
\[u^2_{i,j}=u_{i,j}=u^*_{i,j}\mathrm{~and~}\sum_{k=1}^N u_{i,k}=\sum_{k=1}^N u_{k,j}=1_A~\mathrm{for~all~}1\leq i,j\leq N.\]

Also, we define a comultiplication $\Delta:A\rightarrow A\otimes_{min}A$ by $\displaystyle  u_{i,j}\mapsto \sum_{k=1}^N u_{i,k}\otimes u_{k,j}$. Then $(A,\Delta)$ forms a compact quantum group called the Free permutation quantum group. We denote it by $S_N^+$.
\end{definition}

These free quantum groups are of Kac type, so that the Haar states are tracial states. Also, for all $N\geq 2$, $\mathrm{Irr}(O_N^+)$ and $\mathrm{Irr}(S_{N+2}^+)$ can be identified with $\left \{0\right\}\cup \n$. Moreover,
\[n_k=\left \{ \begin{array}{ll} k+1& \mathrm{if}~\g=O_2^+~\mathrm{or}~S_4^+\\ \frac{r_0}{2r_0-N}r_0^{k+1}+\frac{N-r_0}{2r_0-N}(N-r_0)^{k+1}& \mathrm{if}~\g= O_N^+\mathrm{~or~}S_{N+2}~\mathrm{with~}N\geq 3\end{array} \right.\]
where $r_0$ is the largest solution of the equation $X^2-NX+1=0$. Note that $n_k\approx r_0^k$ if $N\geq 3$.

\subsection{The noncommutative Marcinkiewicz interpolation theorem}

The classical Marcinkiewicz interpolation theorem has a natural non-commutative analogue for semi-finite von Neumann algebras.

\begin{theorem}(The non-commutative Marcinkiewicz interpolation theorem \cite{Xu07})\label{non-comm.Marcin.}

Let $M$ be equipped with a normal semifinite faithful trace $\phi$ and $1\leq p_1<p<p_2< \infty$. Assume that a sub-linear map $A:M\rightarrow L^1(N)$ satisfies the following: There exists $C_1,C_2>0$ such that for any $T_1\in L^{p_1}(M),T_2\in L^{p_2}(M)$ and for any $y>0$,
\begin{equation}\label{eq10}
\phi(1_{(y,\infty)}(|AT_1|))\leq (\frac{C_1}{y})^{p_1}||T_1||_{L^{p_1}(M)}^{p_1},
\end{equation}
\[\phi(1_{(y,\infty)}(|AT_2|))\leq (\frac{C_2}{y})^{p_2}||T_2||_{L^{p_2}(M)}^{p_2}.\]
Then $A:L^p(M)\rightarrow L^p(N)$ is a bounded map.

\end{theorem}

\begin{proof}

The proof of [Theorem 1.22, \cite{Xu07}] is still valid under a natural modification. Also, the direct sum $K:=M\oplus N$ and natural extension $\widetilde{A}:K\rightarrow L^1(K)$ gives another proof.
\end{proof}

If the sub-linear operator $A$ satisfies the inequality (\ref{eq10}), then we say that $A$ is of weak type $(p_1,p_1)$. Also, the boundedness of $A:L^p(M)\rightarrow l^p(N)$ implies that $A$ is of weak type $(p,p)$.

Now denote by $c(\mathrm{Irr}(\g),\nu)$ the space of all functions on the discrete space $\mathrm{Irr}(\g)$ with a positive measure $\nu$. Then the above theorem is written as follows:

\begin{corollary}\label{Mar3}

Let $\g$ be a Kac type compact quantum group and let $1\leq p_1<p<p_2<\infty$. Assume that $A:L^{\infty}(\g)\rightarrow c(\mathrm{Irr}(\g),\nu)$ is sub-linear and satisfies the following: There exists $C_1,C_2>0$ such that for any $T_1\in L^{p_1}(\g),T_2\in L^{p_2}(\g)$ and for any $y>0$,
\[\sum_{\alpha:~|(AT_1)(\alpha)|\geq y}\nu(\alpha)\leq (\frac{C_1}{y})^{p_1}||T_1||_{L^{p_1}(\g)}^{p_1},\]
\[\sum_{\alpha:~|(AT_2)(\alpha)|\geq y}\nu(\alpha)\leq (\frac{C_2}{y})^{p_2}||T_2||_{L^{p_2}(\g)}^{p_2}.\]
Then $A:L^p(\g)\rightarrow \ell^p(\mathrm{Irr}(\g),\nu)$ is a bounded map.
\end{corollary}

\section{Paley-type inequalities}
\subsection{General Approach}

In this subsection, we derive a Paley-type inequality for Kac type compact quantum group $\g$ via fundamental techniques such as Hausdorff-Young inequality, Plancherel theorem and the non-commutative Marcinkiewicz interpolation theorem.

We prove the following theorem by adapting techniques used in \cite{ANR15}.

\begin{theorem}\label{paley-CQG}
Let $\g$ be a Kac type compact quantum group and let $w:\mathrm{Irr}(\g)\rightarrow (0,\infty)$ be a function such that $C_w:=\displaystyle \sup_{t>0}\left \{t\cdot \sum_{\alpha:~w(\alpha)\geq t}n_{\alpha}^2\right\}<\infty$. Then, for each $1<p\leq 2$, there exists a universal constant $K=K(p)>0$ such that 
\begin{equation}\label{eq1}
(\sum_{\alpha\in \mathrm{Irr}(\g)}w(\alpha)^{2-p}n_{\alpha}^{2-\frac{p}{2}}||\widehat{f}(\alpha)||_{HS}^p)^{\frac{1}{p}}\leq K ||f||_{L^p(\g)}
\end{equation}
for all $f\sim \displaystyle \sum_{\alpha\in \mathrm{Irr}(\g)}n_{\alpha}\mathrm{tr}(\widehat{f}(\alpha)u^{\alpha})\in L^p(\g)$.
\end{theorem}

\begin{proof}

Put $\nu(\alpha):=w(\alpha)^2n_{\alpha}^2$. We will show that the sub-linear operator $A:L^1(\g)\rightarrow c(\mathrm{Irr}(\g),\nu)$, $f\mapsto \displaystyle (\frac{||\widehat{f}(\alpha)||_{HS}}{\sqrt{n_{\alpha}}w(\alpha)})_{\alpha\in \mathrm{Irr}(\g)}$, is a well-defined bounded map from $L^p(\g)$ into $\ell^p(\mathrm{Irr}(\g),\nu)$ for all $1<p\leq 2$.

First, 
\begin{align*}
\sum_{\alpha\in \mathrm{Irr}(\g)}||(Af)(\alpha)||_{HS}^2\nu(\alpha)&=\sum_{\alpha\in \mathrm{Irr}(\g)}n_{\alpha}||\widehat{f}(\alpha)||_{HS}^2=||f||_{L^2(\g)}^2.
\end{align*}
This implies that $A$ is of (strong) type $(2,2)$ with $C_2=1$.

Second, for all $y>0$, since $\displaystyle \frac{||\widehat{f}(\alpha)||_{HS}}{\sqrt{n_{\alpha}}}=(\frac{\mathrm{tr}(\widehat{f}(\alpha)^*\widehat{f}(\alpha))}{n_{\alpha}})^{\frac{1}{2}}\leq ||\widehat{f}(\alpha)||\leq ||f||_{L^1(\g)},$
\begin{align*}
\sum_{\alpha:~||Af(\alpha)||_{HS}\geq y}\nu(\alpha)&\leq \sum_{\alpha:~w(\alpha) \leq \frac{||f||_1}{y}}w(\alpha)^2n_{\alpha}^2\\
&=\sum_{\alpha:~w(\alpha)\leq \frac{||f||_1}{y}}\int_0^{w(\alpha)^2}n_{\alpha}^2 dx\\
&=\int_0^{(\frac{||f||_1}{y})^2}[\sum_{x^{\frac{1}{2}}\leq w(\alpha)\leq \frac{||f||_1}{y}}n_{\alpha}^2]dx\mathrm{~by~the~Fubini~theorem}\\
&=2\int_0^{\frac{||f||_1}{y}}t(\sum_{\alpha:~t\leq w(\alpha)\leq \frac{||f||_1}{y}}n_{\alpha}^2)dt~\mathrm{by~substituting~}x~to~t^2\\
&\leq 2 C_w \frac{||f||_1}{y}.
\end{align*}
This says that $A$ is of weak type $(1,1)$ with $C_1=2 C_w$.

Now, by Corollary \ref{Mar3}, 
\[[\sum_{\alpha\in \mathrm{Irr}(\g)}w(\alpha)^{2-p}n_{\alpha}^{p(\frac{2}{p}-\frac{1}{2})}||\widehat{f}(\alpha)||_{HS}^p]^{\frac{1}{p}}\lesssim ||f||_{L^p(\g)}.\]

\end{proof}

The left hand side of the inequality (\ref{eq1}) can be reduced to a more familiar form. Recall that the natural non-commutative $\ell^p$-norm on $\ell^{\infty}(\widehat{\g})=\ell^{\infty}-\oplus_{\alpha\in \mathrm{Irr}(\g)} M_{n_{\alpha}}$ is given by
\[||(A_{\alpha})_{\alpha\in \mathrm{Irr}(\g)}||_{\ell^p(\widehat{\mathbb{G}})}=(\sum_{\alpha\in \mathrm{Irr}(\g)}n_{\alpha}||A_{\alpha}||_{S^p_{n_{\alpha}}}^p)^{\frac{1}{p}}\]
under the condition that $\g$ is of Kac type.

\begin{corollary}\label{paley-CQG2}

Let $1<p\leq 2$ and $w$ be a function satisfying the condition of Theorem \ref{paley-CQG}. Then we have that
\[(\sum_{\alpha\in \mathrm{Irr}(\g)}w(\alpha)^{2-p}n_{\alpha}||\widehat{f}(\alpha)||^p_{S^p_{n_{\alpha}}}))^{\frac{1}{p}}\leq K ||f||_{L^p(\g)}\]
for all $f\sim \displaystyle \sum_{\alpha\in \mathrm{Irr}(\g)}n_{\alpha}\mathrm{tr}(\widehat{f}(\alpha)u^{\alpha})\in L^p(\g)$.
\end{corollary}
\begin{proof}

First,
\[\mathrm{tr}(|\widehat{f}(\alpha)|^p)=||\widehat{f}(\alpha)||_{S^p_{n_{\alpha}}}^p.\]

Put $\displaystyle \frac{1}{r}=\frac{1}{p}-\frac{1}{2}$. Then $2<r\leq \infty$ and
\[\mathrm{tr}(|\widehat{f}(\alpha)|^p)\leq ||\widehat{f}(\alpha)||_{HS}^{p}||Id_{n_{\alpha}}||_{S^r_{n_{\alpha}}}^{p}=n_{\alpha}^{1-\frac{p}{2}}||\widehat{f}(\alpha)||_{HS}^p.\]

This completes proof easily.
\end{proof}

Now we discuss an important subclass of compact quantum groups, namely compact matrix quantum groups which allows the natural length function on $\mathrm{Irr}(\g)$.

\begin{definition}\label{CMQG1}

A compact matrix quantum group is given by a pair $(A,\Delta,u)$ with a unital $C^*$-algebra $A$, a $*$-homomorphism $\Delta:A\rightarrow A\otimes_{min}A$ and a unitary $u=(u_{i,j})_{1\leq i,j\leq n}\in M_n(A)$ such that (1) $\Delta:u_{i,j}\mapsto \displaystyle \sum_{k=1}^n u_{i,k}\otimes u_{k,j}$, (2) $\overline{u}=(u_{i,j}^*)_{1\leq i,j\leq n}$ is invertible in $M_n(A)$ and (3) $\left \{u_{i,j}\right\}_{1\leq i,j\leq n}$ generates $A$ as a $C^*$-algebra.
\end{definition}

By definition, free orthogonal quantum groups $O_N^+$ and free permutation quantum groups $S_N^+$ are compact matrix quantum groups. Also, in the class of compact quantum groups, the subclass of compact matrix quantum groups is characterized by the following proposition. The conjugate $\overline{\alpha}\in Irr(\g)$ of $\alpha\in Irr(\g)$ is determined by $u^{\overline{\alpha}}:=\displaystyle Q_{\alpha}^{\frac{1}{2}}\overline{u^{\alpha}}Q_{\alpha}^{-\frac{1}{2}}=((Q_{\alpha})^{\frac{1}{2}}_{i,i}(u^{\alpha}_{i,j})^*(Q_{\alpha})_{j,j}^{-\frac{1}{2}})_{1\leq i,j\leq n_{\alpha}}$.

\begin{proposition}\label{CMQG2}(\cite{Ti08})

A compact quantum group is a compact matrix quantum group if and only if there exists a finite set $S:=\left \{\alpha_1,\cdots,\alpha_n\right\}\subseteq \mathrm{Irr}(\g)$ such that any $\alpha\in \mathrm{Irr}(\g)$ is contained in some iterated tensor product of elements $\alpha_1,\overline{\alpha_1},\cdots, \alpha_n,\overline{\alpha_n}$ and the trivial corepresentation.
\end{proposition}

Then there is a natural way to define a length function on $\mathrm{Irr}(\g)$ (\cite{Ver07}). For non-trivial $\alpha\in \mathrm{Irr}(\g)$, the natural length $|\alpha|$ is defined by 
\[\min \left \{m\in \left \{0\right\}\cup \n:\exists \beta_1,\cdots,\beta_m\mathrm{~such~that~}\alpha\subseteq \beta_1\otimes \cdots\otimes \beta_m,~\beta_j\in \left \{\alpha_k,\overline{\alpha_k}\right\}_{k=1}^n\right\}.\]
The length of the trivial corepresentation is defined by $0$. 

Then we can extract explicit inequalities from Theorem \ref{paley-CQG} and Corollary \ref{paley-CQG2} by inserting geometric information of underlying quantum group, namely growth rate that is estimated by the quantities $b_k:=\displaystyle \sum_{|\alpha|\leq k}n_{\alpha}^2$ \cite{BV09}.

\begin{corollary}\label{poly1}
Let a Kac type compact matrix quantum group $\g$ satisfy 
\[ b_k=\sum_{\alpha\in \mathrm{Irr}(\g):|\alpha|\leq k}n_{\alpha}^2\leq C(1+k)^{\gamma}\mathrm{~for~all~}k\geq 0\mathrm{~with~}C,\gamma>0\]
with respect to the natural length function. Then, for each $1<p\leq 2$, there exists a universal constant $K=K(p)$ such that
\begin{align}
&(\sum_{\alpha\in \mathrm{Irr}(\g)}\frac{1}{(1+|\alpha|)^{(2-p)\gamma}}n_{\alpha}||\widehat{f}(\alpha)||_{S^p_{n_{\alpha}}}^p)^{\frac{1}{p}}\\
\notag \leq& (\sum_{\alpha\in \mathrm{Irr}(\g)} \frac{1}{(1+|\alpha|)^{(2-p)\gamma}}n_{\alpha}^{2-\frac{p}{2}}||\widehat{f}(\alpha)||_{HS}^p)^{\frac{1}{p}}\leq K||f||_{L^p(\g)}
\end{align}
for all $f\sim \displaystyle \sum_{\alpha\in \mathrm{Irr}(\g)}n_{\alpha}\mathrm{tr}(\widehat{f}(\alpha)u^{\alpha})\in L^p(\g)$.
\end{corollary}
\begin{proof}

Consider a weight function $w(\alpha):=\displaystyle \frac{1}{(1+|\alpha|)^{\gamma}}$. Then
\[\sup_{t>0}\left \{ t\cdot \sum_{\alpha:|\alpha|\leq t^{-\frac{1}{\gamma}}-1}n_{\alpha}^2 \right\}=\sup_{0<t\leq 1}\left \{ t\cdot \sum_{\alpha:|\alpha|\leq t^{-\frac{1}{\gamma}}-1}n_{\alpha}^2 \right\}\leq C \sup_{0<t\leq 1}t\cdot  (t^{-\frac{1}{\gamma}})^{\gamma}= C.\]

Now the conclusion is obtained by Theorem \ref{paley-CQG} and Proposition \ref{paley-CQG2}.

\end{proof}

\subsection{A paley-type inequality under the rapid decay property}

In this subsection, we still assume that $\g$ is a compact matrix quantum group of Kac type. One of the main observations of this paper is that the more detailed geometric information actually improves Theorem \ref{paley-CQG} and Corollary \ref{paley-CQG2} in various ``exponentially growing'' cases. A more refined paley-type inequality can be obtained under the condition that $\g$ has the rapid decay property in the sense of \cite{Ver07}.

\begin{definition}(\cite{Ver07})

Let $\g$ be a Kac type compact matrix quantum group. Then we say that $\g$ has the rapid decay property with respect to the natural length function on $\mathrm{Irr}(\mathbb{G})$ if there exist $C,\beta>0$ such that
\begin{equation}\label{RD}
||\sum_{\alpha\in \mathrm{Irr}(\g):|\alpha|=k}\sum_{i,j=1}^{n_{\alpha}}a^{\alpha}_{i,j}u^{\alpha}_{i,j}||_{L^{\infty}(\g)}\leq C(1+k)^{\beta} ||\sum_{\alpha\in \mathrm{Irr}(\g):|\alpha|=k}\sum_{i,j=1}^{n_{\alpha}}a^{\alpha}_{i,j}u^{\alpha}_{i,j}||_{L^2(\g)}
\end{equation}
for any $k\geq 0$ and scalars $a^{\alpha}_{i,j}\in \Comp$.
\end{definition}

\begin{notation}
\begin{enumerate}
\item When the natural length function on $\mathrm{Irr}(\g)$ is given, we will use the notations $S_k:=\left \{\alpha\in \mathrm{Irr}(\g):|\alpha|=k\right\}$ and $s_k:=\sum_{\alpha\in S_k}n_{\alpha}^2$.
\item We denote by $p_k$ the orthogonal projection from $L^2(\g)$ to the clousre of $\mathrm{span}(\left \{u^{\alpha}_{i,j}:~\alpha\in S_k,~1\leq i,j\leq n_{\alpha}\right\})$.
\end{enumerate}
\end{notation}

\begin{proposition}\label{prop.rd.}

Suppose that a Kac type compact matrix quantum group $\g$ has the rapid decay property with respect to the natural length function on $\mathrm{Irr}(\mathbb{G})$ and with inequality $(\ref{RD})$. Then we have that
\begin{equation}\label{Tool1}
\sup_{k\geq 0}\frac{(\sum_{\alpha\in \mathrm{Irr}(\g):|\alpha|=k}n_{\alpha}||\widehat{f}(\alpha)||_{HS}^2)^{\frac{1}{2}}}{(k+1)^{\beta}}\leq C ||f||_{L^1(\g)}~\mathrm{for~all~}f\in L^1(\g).
\end{equation}
\end{proposition}

\begin{proof}
Since $L^1(\g)$ is isometrically embedded into the dual space $M(\g):=C_r(\g)^*$ and $Pol(\g)$ is dense in $C_r(\g)$, we have that
\begin{align}
\notag ||f||_{L^1(\g)}&=\sup_{x\in Pol(\g):||x||_{L^{\infty}(\g)}\leq 1}\langle f,x\rangle_{L^1(\g),L^{\infty}(\g)}\\
\notag&=\sup_{x\in Pol(\g):||x||_{L^{\infty}(\g)}\leq 1}\langle f,x^*\rangle_{L^1(\g),L^{\infty}(\g)}\\
\notag&=\sup_{x\in Pol(\g):||x||_{L^{\infty}(\g)}\leq 1}\sum_{\alpha\in \mathrm{Irr}(\g)}n_{\alpha}\mathrm{tr}(\widehat{f}(\alpha)\widehat{x}(\alpha)^*)\\
\label{tool2}&\geq \sup_{x\in Pol(\g):\sum_{k\geq 0}C(k+1)^{\beta}(\sum_{\alpha:|\alpha|=k}n_{\alpha}||\widehat{x}(\alpha)||_{HS}^2)^{\frac{1}{2}}\leq 1}\sum_{\alpha\in \mathrm{Irr}(\g)}n_{\alpha}\mathrm{tr}(\widehat{f}(\alpha)\widehat{x}(\alpha)^*)\\
\notag&\geq \sup_{k\geq 0}\sup_{(\sum_{\alpha:|\alpha|=k}n_{\alpha}||\widehat{x}(\alpha)||_{HS}^2)^{\frac{1}{2}}\leq 1}\sum_{\alpha\in \mathrm{Irr}(\g):|\alpha|=k}\frac{n_{\alpha}}{C(k+1)^{\beta}}\mathrm{tr}(\widehat{f}(\alpha)\widehat{x}(\alpha)^*)\\
\notag&=\sup_{k\geq 0}\frac{(\sum_{\alpha\in \mathrm{Irr}(\g):|\alpha|=k}n_{\alpha}||\widehat{f}(\alpha)||_{HS}^2)^{\frac{1}{2}}}{C(k+1)^{\beta}}.
\end{align}
\end{proof}

\begin{theorem}\label{paley-CQG-rd}

Let a Kac type compact matrix quantum group $\g$ have the rapid decay property with respect to the natural length function on $\mathrm{Irr}(\g)$ and with inequality (\ref{RD}). Also, suppose that a weight function $w:\left \{0\right\}\cup \n\rightarrow (0,\infty)$ satisfies
\begin{equation}
C_w:=\sup_{y>0}\left \{y\cdot \sum_{k\geq 0:\frac{(k+1)^{\beta}}{w(k)}\leq \frac{1}{y}}(k+1)^{2\beta} \right\}<\infty.
\end{equation}
Then, for each $1<p\leq 2$, there exists a universal constant $K=K(p)>0$ such that
\begin{equation}
(\sum_{k\geq 0}w(k)^{2-p}(\sum_{\alpha\in \mathrm{Irr}(\g):|\alpha|=k}n_{\alpha}||\widehat{f}(\alpha)||_{HS}^2)^\frac{p}{2})^{\frac{1}{p}}\leq K ||f||_{L^p(\g)}
\end{equation}
for all $f\sim \displaystyle \sum_{\alpha\in \mathrm{Irr}(\g)}n_{\alpha}\mathrm{tr}(\widehat{f}(\alpha)u^{\alpha}) \in L^p(\g)$.
\end{theorem}

\begin{proof}

Put $\nu(k):=\displaystyle w(k)^2$. We will show that the sub-linear operator $A:L^1(\g)\rightarrow c(\left \{0\right\}\cup \n,\nu)$, $f\mapsto \displaystyle (\frac{||p_k(f)||_{L^2(\g)}}{w(k)})_{k\geq 0}$, is a well-defined bounded map from $L^p(\g)$ into $l^p(\left \{0\right\}\cup \n,\nu)$ for all $1<p\leq 2$.

Firstly, 
\begin{align*}
\sum_{k\geq 0}(\frac{||p_k(f)||_{L^2(\g)}}{w(k)})^2\nu(k)&=\sum_{k\geq 0}||p_k(f)||_{L^2(\g)}^2=||f||_{L^2(\g)}^2.
\end{align*}
Therefore, $A$ is of (weak) type $(2,2)$ with $C_2=1$.

Secondly, for all $y>0$, 

\begin{align*}
\sum_{k\geq 0: (Af)(k)>y}\nu(k)&\leq \sum_{k:\frac{w(k)}{(k+1)^{\beta}}<\frac{C||f||_{L^1(\g)}}{y} }w(k)^2 \mathrm{~by~Proposition~\ref{prop.rd.}}.
\end{align*}
Now put $\widetilde{w}(k):=\displaystyle \frac{w(k)}{(k+1)^{\beta}}$. Then
\begin{align*}
&=\sum_{k:\widetilde{w}(k)<\frac{C||f||_{L^1(\g)}}{y}}\int_0^{\widetilde{w}(k)^2}(k+1)^{2\beta}dx\\
&\leq \int_0^{(\frac{C||f||_{L^1(\g)}}{y})^2}\sum_{k:\sqrt{x}\leq \widetilde{w}(k)}(k+1)^{2\beta}dx\\
&=2\int_0^{\frac{C||f||_{L^1(\g)}}{y}}t\cdot \sum_{k:t\leq \widetilde{w}(k)}(k+1)^{2\beta}  dt\mathrm{~by~substituting~}x=t^2\\
&\leq \frac{2C_wC||f||_{L^1(\g)}}{y}.
\end{align*}

Therefore, by Corollary \ref{Mar3}, we obtain that
\[(\sum_{k\geq 0}w(k)^{2-p}||p_k(f)||_{L^2(\g)}^p)^{\frac{1}{p}}\lesssim ||f||_{L^p(\g)}.\]

\end{proof}

\begin{corollary}\label{paley-CQG-RD0}

Let a Kac type compact matrix quantum group $\g$ have the rapid decay property with respect to the natural length function on $\mathrm{Irr}(\g)$ and with inequality (\ref{RD}). Then, for each $1<p\leq 2$, there exists a universal constant $K=K(p)>0$ such that
\begin{equation}
(\sum_{k\geq 0}\frac{1}{(1+k)^{(2-p)(\beta+1)}}(\sum_{\alpha\in \mathrm{Irr}(\g):|\alpha|=k}n_{\alpha}||\widehat{f}(\alpha)||_{HS}^2)^{\frac{p}{2}})^{\frac{1}{p}}\leq K ||f||_{L^p(\g)}
\end{equation}
for all $f\sim \displaystyle \sum_{\alpha\in \mathrm{Irr}(\g)}n_{\alpha}\mathrm{tr}(\widehat{f}(\alpha)u^{\alpha}) \in L^p(\g)$.
\end{corollary}

\begin{proof}
Take $w(k):=\displaystyle \frac{1}{(1+k)^{\beta+1}}$ and $E=\mathrm{Irr}(\g)$. Then
\begin{align*}
C_w&=\sup_{y>0}\left \{y\cdot \sum_{k\geq 0:(1+k)^{2\beta+1}\leq \frac{1}{y}}(1+k)^{2\beta}\right\}\\
&\leq \sup_{0<y\leq 1}\left \{y\cdot \int_{1}^{(\frac{1}{y})^{\frac{1}{2\beta+1}}+1}t^{2\beta}dt\right\}\\
&\leq \sup_{0<y\leq 1}\left \{y\cdot \frac{(2\cdot (\frac{1}{y})^{\frac{1}{2\beta+1}})^{2\beta+1}}{2\beta+1}\right\}\\
&=\frac{2^{\beta+1}}{2\beta+1}<\infty.
\end{align*}
\end{proof}

\begin{corollary}\label{paley-CQG-RD}

Let a Kac type compact matrix quantum group $\g$ have the rapid decay property with respect to the natural length function on $\mathrm{Irr}(\g)$ and with inequality (\ref{RD}). Then, for each $1<p\leq 2$, there exists a universal constant $K=K(p)>0$ such that
\begin{equation}
(\sum_{k\geq 0}\sum_{\alpha\in \mathrm{Irr}(\g):|\alpha|=k}\frac{1}{(1+|\alpha|)^{(2-p)(\beta+1)}(\sum_{\beta\in S_k}n_{\beta}^2)^{\frac{2-p}{2}}}n_{\alpha}||\widehat{f}(\alpha)||_{S^p_{n_{\alpha}}}^p)^{\frac{1}{p}}\leq K ||f||_{L^p(\g)}
\end{equation}
for all $f\sim \displaystyle \sum_{\alpha\in \mathrm{Irr}(\g)}n_{\alpha}\mathrm{tr}(\widehat{f}(\alpha)u^{\alpha}) \in L^p(\g)$.
\end{corollary}

\begin{proof}
Since $\displaystyle \frac{1}{p}=\frac{1}{2}+\frac{2-p}{2p}$ and $n_{\alpha}^{-\frac{1}{p}}||\widehat{f}(\alpha)||_{S^p_{n_{\alpha}}}\leq n_{\alpha}^{-\frac{1}{2}}||\widehat{f}(\alpha)||_{HS}$, we have that
\begin{align*}
\sum_{\alpha\in S_k}n_{\alpha}||\widehat{f}(\alpha)||_{S^p_{n_{\alpha}}}^p&\leq \sum_{\alpha\in S_k}n_{\alpha}^{2-\frac{p}{2}}||\widehat{f}(\alpha)||_{HS}^p\\
&=||(n_{\alpha}^{\frac{1}{2}}||\widehat{f}(\alpha)||_{HS}\cdot n_{\alpha}^{\frac{2}{p}-1})_{\alpha\in S_k}||_{\ell^p(S_k)}^p\\
&\leq ||(n_{\alpha}^{\frac{1}{2}}||\widehat{f}(\alpha)||_{HS})_{\alpha\in S_k}||_{\ell^2(S_k)}^p\cdot (\sum_{\alpha\in S_k}n_{\alpha}^2)^{\frac{2-p}{2}}\\
&=(\sum_{\alpha\in S_k}n_{\alpha}^2)^{\frac{2-p}{2}}\cdot (\sum_{\alpha\in S_k}n_{\alpha}||\widehat{f}(\alpha)||_{HS}^2)^{\frac{p}{2}}.
\end{align*}
Then we obtain the conclusion.
\end{proof}

\section{Hardy-Littlewood inequalities}\label{Sec3}
This section is devoted to establish explicit Hardy-Littlewood inequalities for the main targets: the reduced group $C^*$-algebras $C_r^*(G)$ with finitely generated discrete group $G$, free orthogonal quantum groups $O_N^+$ and free permutation quantum groups $S_N^+$.

\subsection{The reduced group $C^*$-algebras $C_r^*(G)$}

In this subsection, we treat finitely generated discrete groups $G$. As expected, we found clear evidence that the geometric information of the underlying group is of significant importance for understanding non-commutative $L^p$-spaces $L^p(VN(G))$.

\begin{definition}
A discrete group with a fixed finite symmetric generating set $S$ is said to be polynomially growing if there exist $C>0$ and $k>0$ such that
\[\# \left \{g\in G: |g|\leq n \right\}\leq C (1+n)^k~\mathrm{for~all~}n\geq 0.\] 
In this case, the polynomial growth rate $k_0$ is defined as the minimum of such $k$. Then $k_0$ becomes a natural number and is independent of the choice of generating set $S$.
\end{definition}

\begin{theorem}\label{co-comm}
\begin{enumerate}
\item Let $G$ be a finitely generated discrete group, which has the polynomial growth rate $k_0$. Then, for each $1<p\leq 2$, there exists a universal constant $K=K(p)$ such that
\begin{equation}\label{HL-poly.disc.}
(\sum_{g\in G}\frac{1}{(1+|g|)^{(2-p)k_0}}|f(g)|^p)^{\frac{1}{p}}\leq K ||\lambda(f)||_{L^p(VN(G))}
\end{equation}
for all $\lambda(f)\displaystyle \sim \sum_{g\in G}f(g)\lambda_g\in L^p(VN(G))$.

\item Let $G$ be a finitely generated discrete group with 
\[b_k=\displaystyle \#\left \{g\in G:|g|_S\leq k\right\}\leq C r^k\mathrm{~for~all~}k\geq 0,\]
where $|\cdot|_S$ is the natural length function with respect to a finite symmetric generating set $S$. Then, for each $1<p\leq 2$, there exists a universal constant $K=K(p,S)>0$ such that
\begin{equation}\label{HL-exp.disc.}
(\sum_{g\in G}\frac{1}{r^{(2-p)|g|}}|f(g)|^p)^{\frac{1}{p}}\leq K ||\lambda(f)||_{L^p(VN(G))}
\end{equation}
for all $\lambda(f)\sim \displaystyle \sum_{g\in G} f(g)\lambda_g\in L^p(VN(G))$.
\end{enumerate}
\end{theorem}

\begin{proof}

(1) Clear from Corollary \ref{poly1}.

(2) Consider $w(g):=\displaystyle \frac{1}{r^{|g|}}$. Then
\[\sup_{t>0}\left \{ t\cdot \sum_{|g|\leq  log_{r}(\frac{1}{t})}1 \right\}=\sup_{0<t\leq 1}\left \{ t\cdot \sum_{|g|\leq  log_{r}(\frac{1}{t})}1 \right\}\leq C.\]

Then the conclusion follows from Theorem \ref{paley-CQG}
\end{proof}

\begin{remark}
\begin{enumerate}
\item For every finitely generated discrete group, there exist $C, r>0$ such that $b_k\leq C r^k$ for all $k\geq 0$ by the Fekete's subadditivity lemma. Therefore, Theorem \ref{co-comm} covers all finitely generated discrete group.
\item In fact, the inequality (\ref{HL-poly.disc.}) is sharp by Theorem \ref{comp.char.}.
\end{enumerate}
\end{remark}

Although we can always find inequality (\ref{HL-exp.disc.}) for every finitely generated discrete group, we can get a much better result by adding more detailed geometric information of underlying group. Indeed, if we assume hyperbolicity of group, then the inequality is considerably improved.

\begin{theorem}\label{main-gp.v.N.}

Let $G$ be any non-elementary word hyperbolic group with $b_k\leq Cr^k$ for all $k\geq 0$ with respect to a finite symmetric generating set $S$. Then, for each $1<p\leq 2$, there exists a universal constant $K=K(S,p)$ such that 
\begin{align}
(\sum_{g\in G}\frac{1}{r^{\frac{(2-p)|g|}{2}}(1+|g|)^{4-2p}}|f(g)|^p)^{\frac{1}{p}}&\leq (\sum_{k\geq 0}\frac{1}{(k+1)^{4-2p}}(\sum_{g\in G:|g|=k}|f(g)|^2)^{\frac{p}{2}})^{\frac{1}{p}}\\
\notag&\leq K ||\lambda(f)||_{L^p(VN(G))}.
\end{align}
for all $\lambda(f)\sim \displaystyle \sum_{g\in G} f(g)\lambda_g\in L^p(VN(G))$.
\end{theorem}

\begin{proof}
The conclusion follows from Corollary \ref{paley-CQG-RD} and \cite{Har88}.
\end{proof}

\subsection{Free quantum groups}
Let us begin the investigation of `genuine' quantum examples: Free orthogonal quantum groups $O_N^+$ and free permutation quantum groups $S_{N+2}^+$. Moreover, the Hardy-Littlewood inequality for $O_N^+$ becomes an equivalence under centrality, positivity, monotonic decrease and a non-oscillation type condition of Fourier coefficients, as for the result for $SU(2)$ [Theorem 2.10 \cite{ANR15}]. This is considered as a way to prove sharpness of Hardy-Littlewood inequalities.

\begin{theorem}\label{main-Free.Ortho.}
\begin{enumerate}
\item Let $\g$ be the free orthogonal quantum group $O_2^+$ or the free permutation quantum group $S_4^+$. Then, for each $1<p\leq 2$, there exists a universal constant $K=K(p)$ such that
\begin{equation}\label{HL-Free.Orth.-1}
(\sum_{k\geq 0}\frac{1}{(1+k)^{6-3p}}n_k||\widehat{f}(k)||_{S^p_{k+1}}^p)^{\frac{1}{p}}\leq (\sum_{k\geq 0} \frac{1}{(1+k)^{6-3p}}n_k^{2-\frac{p}{2}}||\widehat{f}(k)||_{HS}^p)^{\frac{1}{p}}\leq K||f||_{L^p(\g)}
\end{equation}
for all $f\sim \displaystyle \sum_{k\geq 0}n_k\mathrm{tr}(\widehat{f}(k)u^k)\in L^p(\g)$.

\item Let $\g$ be a free orthogonal quantum group $O_N^+$ or a free permutation quantum group $S_{N+2}^+$ with $N\geq 3$. Then, for each $1<p\leq 2$, there exists a universal constant $K=K(p)$ such that
\begin{align}\label{HL-Free.Ortho.-2}
(\sum_{k\geq 0}\frac{1}{r_0^{(2-p)k}(1+k)^{4-2p}}n_k||\widehat{f}(k)||_{S^p_{n_k}}^p)^{\frac{1}{p}}&\leq (\sum_{k\geq 0}\frac{1}{r_0^{(2-p)k}(1+k)^{4-2p}}n_k^{2-\frac{p}{2}}||\widehat{f}(k)||_{HS}^p)^{\frac{1}{p}}\\
\notag&\leq K ||f||_{L^p(\g)}
\end{align}
for all $f\sim \displaystyle \sum_{k\geq 0}n_k\mathrm{tr}(\widehat{f}(k)u^k)\in L^p(\g)$,
where $\displaystyle r_0=\frac{N+\sqrt{N^2-4}}{2}$.
\end{enumerate}
\end{theorem}
\begin{proof}
(1) In this case, $n_k=k+1$ for all $k$, so that the conclusion follows from Corollary \ref{poly1}.

(2) It is known that free orthogonal quantum groups $O_N^+$ and free quantum groups $S_{N+2}$ with $N\geq 3$ have the rapid decay property with $\beta=1$ (\cite{Ver07} and \cite{Br13}). Also, $s_k=n_k^2\approx r_0^{2k}$ for all $k\in \left \{0\right\}\cup \n$. Therefore, Corollaries \ref{paley-CQG-RD0} and \ref{paley-CQG-RD} complete proof.
\end{proof}

\begin{remark}
All results of this paper for $S_N^+$ can be extended to quantum automorphism group $\g_{aut}(B,\psi)$ with a $\delta$-trace $\psi$ and $dim(B)=N$ via the same proofs.
\end{remark}

An important observation for the free orthogonal quantum groups $O_N^+$ is that the inequalities (\ref{HL-Free.Orth.-1}) and (\ref{HL-Free.Ortho.-2}) become actually equivalences in a large class. Essentially, this fact is based on the result for SU(2) [Theorem 2.10, \cite{ANR15}] and the following lemma moves the result to $O_N^+$.

\begin{lemma}\label{monoidal}
Let $\g=O_N^+$ or $S_{N+2}^+$ with $N\geq 2$ and consider $G=SU(2)$ or $SO(3)$ for each cases. Then, for $f\sim \displaystyle \sum_{n\geq 0} c_n\chi^1_n\in L^p(\g)$, 
\[||f||_{L^p(\g)}=||\Phi(f)\sim \sum_{n\geq 0}c_n\chi^2_n||_{L^p(G)}\mathrm{~for~all~}1\leq p\leq \infty.\]
Here, $\chi^1_n=\mathrm{tr}(u^n)$ and $\chi^2_n=\mathrm{tr}(v^n)$ where $u^n$ and $v^n$ are the $n$-th irreducible unitary representations of $\g$ and $G$ respectively.
\end{lemma}
\begin{proof}
In the above cases, it is known that $\g$ and $G$ share the fusion rule. For details, see [Proposition 6.7, \cite{Wa16}]. Now, for any $x=\displaystyle \sum_{k\geq 0}c_k\chi^1_k\in Pol(\g)$ and $m\in \n$,
\begin{align*}
h((x^*x)^m)&=h((\sum_{k,l\geq 0}\overline{c_k}c_l\chi^1_{k}\chi^1_l)^m)\\
&=\sum_{k_1,l_1,\cdots,k_m,l_m\geq 0}\overline{c_{k_1}\cdots c_{k_m}}c_{l_1}\cdots c_{l_m}h(\chi_{k_1}^1\chi^1_{l_1}\cdots \chi^1_{k_m}\chi^1_{l_m})\\
&=\sum_{k_1,l_1,\cdots,k_m,l_m\geq 0}\overline{c_{k_1}\cdots c_{k_m}}c_{l_1}\cdots c_{l_m}\int_{G}\chi_{k_1}^2\chi^2_{l_1}\cdots \chi^2_{k_m}\chi^2_{l_m}\\
&=\int_{G}(\sum_{k,l\geq 0}\overline{c_k}c_l\chi^2_k\chi^2_l)^m\\
&=\int_{G}|x'|^{2m},
\end{align*}
where $x'=\displaystyle \sum_{k\geq 0}c_k \chi^2_k\in Pol(G)$. Then the Stone-Weistrass theorem completes this proof.
\end{proof}

\begin{corollary}\label{main-equiv.}

Let $N\geq 2$, $\displaystyle \frac{3}{2}< p\leq 2$ and fix $D>0$. Also, assume that $f\sim\displaystyle \sum_{k\geq 0}c_k\chi_k\in L^{\frac{3}{2}}(O_N^+)$ satisfies
\begin{equation}\label{equiv.cond.}
c_k\geq c_{k+1}\geq 0~\mathrm{and~}\sum_{m\geq k}\frac{c_m}{m+1}\leq D\cdot c_k~\mathrm{for~all~}k\geq 0.
\end{equation}
Then we have
\begin{align}\label{equiv.free}
||f||_{L^p(O_N^+)}&\approx (\sum_{k\geq 0}(1+k)^{2p-4}c_k^p)^{\frac{1}{p}}.
\end{align}
\end{corollary}

\section{A strong Hardy-Littlewood inequality}

The studies of Hardy-Littlewood inequalities in \cite{ANR15}, \cite{HL27} and \cite{HR74} deal with general $L^p$-functions, but a plenty of classical results of harmonic analysis on $\tor$ shows that a theorem on a function space can have a stronger form when restricted to ``holomorphic'' setting \cite{KS07}.

An evidnce on non-commutative setting is ``the strong Haagerup inequality'' on the reduced group $C^*$-algebras $C_r^*(\mathbb{F}_N)$. More precisely, it was shown that the rapid decay property can be strengthen in general holomorphic setting \cite{KS07}.

Let $g_1,\cdots,g_N$ be canonical generators of $\mathbb{F}_N$ and denote by $\mathbb{F}_N^+$ the the set of elements of the form $g_{i_1}g_{i_2}\cdots g_{i_m}$ with $m\in \left \{0\right\}\cup \n$ and $1\leq i_k\leq N$ for all $1\leq k\leq m$.

\begin{theorem}(Strong Haagerup inequality for $C_r^*(\mathbb{F}_N)$)

Consider a subset $E:=\mathbb{F}_N^+$ and $E_k:=\left \{g\in E:|g|=k\right\}$. Then, for any $k\in \left \{0\right\}\cup \n$, we have that
\[||\sum_{g\in E_k}f(g)\lambda_g||_{C_r^*(\mathbb{F}_N)}\leq \sqrt{e}\sqrt{k+1}(\sum_{g\in E_k}|f(g)|^2)^{\frac{1}{2}}.\]
\end{theorem}

Based on this information, we can modify the inequality (\ref{Tool1}) as follows.
\begin{proposition}
Let $N\geq 2$. Then we have that
\begin{equation}\label{Tool2}
||f||_{A(\mathbb{F}_N)}\geq \frac{1}{\sqrt{e}}\sup_{k\geq 0}\frac{(\sum_{g\in E_k}|f(g)|^2)^{\frac{1}{2}}}{(k+1)^{\frac{1}{2}}}~\mathrm{for~all~}f\in A(\mathbb{F}_N).
\end{equation}
\end{proposition}

\begin{proof}
We can repeat the proof of Proposition \ref{prop.rd.}. The only difference is the improvement of $(1+k)^{\beta}$ to $(1+k)^{\frac{1}{2}}$ in inequality (\ref{tool2}). Then we are able to get conclusion by restricting support of $x\in C_c(G)$ to $\mathbb{F}_N^+$ in the proof.
\end{proof}

\begin{theorem}\label{strong1}
Let $N\geq 2$. Then, for each $1<p\leq 2$, there exists a universal constant $K=K(p)>0$ such that
\begin{align}
(\sum_{g\in \mathbb{F}_N}\frac{1}{(1+|g|)^{\frac{3}{2}(2-p)}N^{\frac{(2-p)|g|}{2}}}|f(g)|^p)^{\frac{1}{p}}&\leq (\sum_{k\geq 0}\frac{1}{(1+k)^{\frac{3}{2}(2-p)}}(\sum_{g\in \mathbb{F}_N:|g|=k}|f(g)|^2)^{\frac{p}{2}})^{\frac{1}{p}}\\
\notag&\leq K||\lambda(f)||_{L^p(VN(\mathbb{F}_N))}
\end{align}
for all $\lambda(f)\sim \displaystyle \sum_{g\in \mathbb{F}_N}f(g)\lambda_g\in L^p(VN(\mathbb{F}_N))$ with $supp(f)\subseteq \mathbb{F}_N^+$.
\end{theorem}
\begin{proof}
It can be also obtained by repeating the proof of Theorem \ref{paley-CQG-rd} and Corollary \ref{paley-CQG-RD}. The only difference is to replace the operator $A$ with $\displaystyle \lambda(f)\mapsto (\frac{||f\cdot \chi_{E_k}||_{l^2(G))}}{w(k)^2})_{k\geq 0}$ and $(1+k)^{\beta}$ with $(1+k)^{\frac{1}{2}}$. Also, we choose a  weight function $w$ on $\left \{0\right\}\cup \n$ by $w(k):=\displaystyle \frac{1}{(1+k)^{\frac{3}{2}}}$. Then we can derive new inequality for general $\lambda(f)\in L^p(VN(\mathbb{F}_N))$, but our consideration is in the case $supp(f)\subseteq \mathbb{F}_N^+$.
\end{proof}

\section{Sharpness}\label{Sharp}

The studied Hardy-Littlewood inequalities give a decay pair $(r,s)$ such that the multiplier
\[\mathcal{F}_{w_{r,s}}:L^p(\g)\rightarrow \ell^p(\widehat{\g}),~f\mapsto (w_{r,s}(\alpha)\widehat{f}(\alpha))_{\alpha\in Irr(\g)},\]
is bounded for each cases, where $w_{r,s}(\alpha)=\displaystyle \frac{1}{r^{|\alpha|}(1+|\alpha|)^s}$ with respect to the natural length $|\cdot|$ on Irr($\g$). Here is the list: $\displaystyle(0,\frac{n(2-p)}{p})$ for $C(G)$ with compact Lie group $G$, $\displaystyle(0,\frac{k_0(2-p)}{p})$ for $C_r^*(G)$ with polynomially growing discrete group $G$, $\displaystyle(0,\frac{3(2-p)}{p})$ for $O_2^+$ or $S_4^+$ and $\displaystyle (r_0^{\frac{2-p}{p}},\frac{2(2-p)}{p})$ for $O_N^+$ or $S_{N+2}$ with $N\geq 3$.

\begin{remark}\label{rmk1}(\cite{LY15} and \cite{Wa73})
If $G$ is a compact Lie group, then $\sqrt{\kappa_{\pi}}$ is equivalent to the natural length function $||\cdot ||_S$ generated by the fundamental generating set $S$ of $\widehat{G}$. Equivalently, $(1+\kappa_{\pi})^{\frac{\beta}{2}}\approx (1+||\pi||_S)^{\beta}$. 
\end{remark}

To assert that the established inequalities are sharp, we will show that there is no slower decay pair $(r,s)$ such that $\mathcal{F}_{w_{r,s}}$ is bounded for the above cases.

This viewpoint is different from the spirit of [Theorem 2.10 ,\cite{ANR15}] or Theorem \ref{main-equiv.}, which requires finding an equivalence on a subclass. However, our approach is quite natural since it is strongly related to Sobolev embedding theorem. For example, $\mathcal{F}_{w_{0,s}}:L^p(\tor^d)\rightarrow l^p(\z^d)$ is bounded if and only if $H^s_p(\tor^d)\subseteq L^{p'}(\tor^d)$ if and only if $H^{\frac{ps}{2-p}(\frac{1}{q}-\frac{1}{r})}_q(\tor^d)\subseteq L^r(\tor^d)$ for all $1<q<r<\infty$, where $H^s_p$ is the Bessel potential space. In this direction, we will provide a Sobolev embedding type interpretation for results of this section in subsection \ref{sobolev}.

In addition, this view has a definite advantage over looking for equivalence because we can cover much larger class.

Our first strategy is handling an ultracontractivity problem on $C(G)$ with compact Lie groups, $C_r^*(G)$ with polynomially growing discrete group. Actually ultracontractivity problem is strongly related to Sobolev embedding property \cite{Xi16}.

Let $M$ be the von Neumann subalgebra generated by $\left \{\chi_{\alpha}\right\}_{\alpha\in \mathrm{Irr}(\g)}$ in $L^{\infty}(\g)$ and consider $L^p(M)$ as the non-commutative $L^p$-space associated to the restriction of the Haar state on $M$. Now suppose that $l:\mathrm{Irr}(\g)\rightarrow (0,\infty)$ is a positive function and there exist $1<p<2$ and a universal constant $C>0$ such that
\begin{equation}\label{sobolev1}
||J(f)\sim \sum_{\alpha\in \mathrm{Irr}(\g)}\frac{1}{l(\alpha)^{\frac{\beta}{p}}}c_{\alpha} \chi_{\alpha}||_{L^{p'}(M)}\leq C||f\sim \sum_{\alpha\in \mathrm{Irr}(\g)}c_{\alpha}\chi_{\alpha}||_{L^p(M)},
\end{equation}
where $J$ is a densely defined positive operator on $L^2(M)$ which maps $\chi_{\alpha}\mapsto \displaystyle \frac{1}{l(\alpha)^{\frac{\beta}{p}}}\chi_{\alpha}$ for all $\alpha\in \mathrm{Irr}(\g)$, $1\leq i,j\leq n_{\alpha}$. Indeed, $J=K^*K$ where $K:\chi_{\alpha}\mapsto \displaystyle \frac{1}{l(\alpha)^{\frac{\beta}{2p}}}\chi_{\alpha}$.

Now take $\phi(t):=t^{\frac{2\beta}{2-p}}$, $\psi(z):=z^{\frac{2\beta}{2-p}}$ and $L:=J^{-\frac{p}{2\beta}}$. Then [Theorem 1.1, \cite{Xi16}] says that there exists a universal constant $C'>0$ such that
\begin{equation}\label{ultra}
||e^{-tL}(f)\sim \sum_{\alpha\in \mathrm{Irr}(\g)}\frac{c_{\alpha}}{e^{tl(\alpha)^{\frac{1}{2}}}}\chi_{\alpha}||_{L^{\infty}(M)}\leq C'\frac{||f||_{L^1(M)}}{t^{\frac{2\beta}{2-p}}}.
\end{equation}
for all $0<t<\infty$ and all $f\sim \displaystyle \sum_{\alpha\in \mathrm{Irr}(\g)}c_{\alpha}\chi_{\alpha} \in L^1(M)$.

Our claim is that we get the following observations in the above situation:

\begin{equation}\label{want}
\sup_{0<t<\infty}\left \{t^{\frac{2\beta}{2-p}}\sum_{\alpha\in \mathrm{Irr}(\g)}\frac{n_{\alpha}^2}{e^{tl(\alpha)^{\frac{1}{2}}}}\right\}=:C<\infty.
\end{equation}
if $\g$ is given by $C(G)$ with compact Lie group or $C_r^*(G)$ with polynomially growing discrete group.

\begin{lemma}([Lemma 4.1, \cite{DR14}] and [Proposition 5.7, \cite{LY15}])\label{conv.cond.}
\begin{enumerate}
\item Let $G$ be a compact Lie group with dimension $n$. Then $\displaystyle \sum_{\pi\in \widehat{G}}\frac{n_{\pi}^2}{(1+\kappa_{\pi})^{\frac{s}{2}}}<\infty$ if and only if $s>n$.
\item Let $G$ be a finitely generated discrete group with polynomial growth rate $k_0$. Then $\displaystyle \sum_{g\in G} \frac{1}{(1+|g|)^{s}}<\infty$ if and only if $s>k_0$.
\end{enumerate}
\end{lemma}

\begin{lemma}\label{tool3}
\begin{enumerate}
\item Let $G$ be a compact Lie group. Then there exist probability measures $\left \{\nu_{t}\right\}_{t>0}$ such that $\widehat{\nu_t}(\pi)=\displaystyle \frac{1}{e^{t\kappa_{\pi}}}Id_{n_{\pi}}$ for all $\pi\in \widehat{G}$. Moreover, $\left \{\nu_{t}\right\}_{t>0}\subseteq L^1(G)$.
\item Let $G$ be a compact Lie group and let $f\sim \displaystyle \sum_{\pi\in \widehat{G}}n_{\pi}\mathrm{tr}(\widehat{f}(\pi)\pi)\in L^{\infty}(G)$ such that $\widehat{f}(\pi)\geq 0$ for all $\pi\in \widehat{G}$. Then
\[||f||_{L^{\infty}(G)}=\sum_{\pi\in \widehat{G}}n_{\pi}\mathrm{tr}(\widehat{f}(\pi)).\]
\item Let $G$ be a discrete group. Then $A(G)$ has a bounded approximate identity if and only if $G$ is amenable. In this case, the bounded approximate identity can be chosen as positive and compactly supported functions on $G$.
\item If $G$ is an amenable discrete group, we have that
\[||\lambda(f)\sim \sum_{g\in G}f(g)\lambda_g||_{VN(G)}=\sum_{g\in G}f(g)\]
for any positive function $f\in l^1(G)$.
\end{enumerate}
\end{lemma}
\begin{proof}
(1) Since $\displaystyle \sum_{\pi\in \widehat{G}}\frac{n_{\pi}^2}{e^{t\kappa_{\pi}}}<\infty$ by Lemma \ref{conv.cond.}, we know that $\nu_t\in A(G)\subseteq C(G)\subseteq L^1(G)$. The family $\left \{\nu_t\right\}_{t>0}$ is called the Heat semigroup of measures.

(2) Since $f\mapsto \mu_t*f$ is a contractive map on $L^{\infty}(G)$ for all $t>0$ where $*$ is the convolution product, we have
\begin{align*}
||f||_{L^{\infty}(G)}&\geq \sup_{t>0}||f_t\sim \sum_{\pi\in \widehat{G}}\frac{n_{\pi}}{e^{t\kappa_{\pi}}}\mathrm{tr}(\widehat{f}(\pi)\pi)||_{C(G)}
\end{align*}
Here, since $\displaystyle \sum_{\pi}\frac{n_{\pi}}{e^{t\kappa_{\pi}}}\mathrm{tr}(\widehat{f}(\pi))\leq ||f||_{L^1(G)}\sum_{\pi}\frac{n_{\pi}^2}{e^{t\kappa_{\pi}}}<\infty$ by Lemma \ref{conv.cond.}, the Fourier series of $f_t$ uniformly converges to $f_t\in C(G)$. Therefore,
\begin{align*}
||f||_{L^{\infty}(G)}&\geq \sup_{t>0}f_t(1)=\sup_{t>0}\sum_{\pi\in \widehat{G}}\frac{n_{\pi}}{e^{t\kappa_{\pi}}}\mathrm{tr}(\widehat{f}(\pi))=\sum_{\pi\in \widehat{G}}n_{\pi}\mathrm{tr}(\widehat{f}(\pi)).
\end{align*}
The other direction is trivial.

(3) See [Theorem 7.1.3, \cite{Ru02}] and its proof. We also may assume the compact supportness by considering $f_{\epsilon}:=f\cdot \chi_{\left \{g\in G:f(g)>\epsilon\right\}}$ for positive $f\in l^1(G)$.

(4) This is the Kesten's condition that is equivalent to amenability.
\end{proof}

Now we can show that the claim is true.

\begin{proposition}\label{claim1}
Let $\g$ be $C(G)$ with compact Lie group or $C_r^*(G)$ with polyomially growing discrete group. Also, suppose that the inequality (\ref{sobolev1}) holds. Then 
\begin{equation}\label{tool10}
\sup_{0<t<\infty}\left \{t^{\frac{2\beta}{2-p}}\sum_{\alpha\in \mathrm{Irr}(\g)}\frac{n_{\alpha}^2}{e^{tl(\alpha)^{\frac{1}{2}}}}\right\}=:C<\infty.
\end{equation}
\end{proposition}

\begin{proof}
(1) By Lemma \ref{tool3},
\begin{align*}
\sum_{\pi\in \widehat{G}}\frac{n_{\pi}^2}{e^{tl(\pi)^{\frac{1}{2}}}}&=\sup_{r>0}\sum_{\pi\in \widehat{G}}\frac{n_{\pi}^2}{e^{tl(\pi)^{\frac{1}{2}}}e^{r\kappa_{\pi}}}\\ &=\sup_{r>0}||e^{-tL}(\nu_r)||_{L^{\infty}(G)}\\
&\leq \frac{C'}{t^{\frac{2\beta}{2-p}}}\sup_{r>0}||\nu_r||_{L^1(G)}\leq \frac{C'}{t^{\frac{2\beta}{2-p}}}
\end{align*}
for all $0<t<\infty$.

(2) There exists a bounded approximate identity $(e_i)_i$ in $A(G)$ that consists of positive and compactly supported functions since polynomially growing discrete group is always amenable. Then inequality (\ref{ultra}) says that
\[\sum_{g\in G}\frac{1}{e^{tl(g)^{\frac{1}{2}}}}=\sup_i\sum_{g\in G}\frac{e_i(g)}{e^{tl(g)^{\frac{1}{2}}}}=\sup_i ||\sum_{g\in G}\frac{e_i(g)}{e^{tl(g)^{\frac{1}{2}}}}\lambda_g||_{C_r^*(G)}\leq \frac{C''}{t^{\frac{2\beta}{2-p}}}\]
since $\displaystyle \lim_i e_i(g)= 1$ for all $g\in G$.

\end{proof}

Proposition \ref{claim1} allows us to extract an interesting quantitive observation.

\begin{proposition}\label{tool4}
Let $\g$ be $C(G)$ with compact Lie group or $C_r^*(G)$ with polyomially growing discrete group. Also, suppose that the inequality (\ref{sobolev1}) holds. Then we have that
\begin{equation}\label{want2}
\sum_{\alpha\in \mathrm{Irr}(\g)}\frac{n_{\alpha}^2}{l(\alpha)^{\frac{m}{2}}}<\infty~\mathrm{for~all~natural~numbers~}m>\frac{2\beta}{2-p}.
\end{equation}
\end{proposition}

\begin{proof}
Choose $\gamma\in (\max \left \{ \frac{2\beta}{2-p},m-1\right\},m)$. Then we have
\[\sup_{0<t\leq 1}\left \{ t^{\gamma}\sum_{\alpha\in \mathrm{Irr}(\g)}\frac{n_{\alpha}^2}{e^{tl(\alpha)^{\frac{1}{2}}}} \right\}=:C_0<\infty\]
from (\ref{tool10}), so that
\[\int_t^1 \sum_{\alpha\in \mathrm{Irr}(\g)}\frac{n_{\alpha}^2}{e^{xl(\alpha)^{\frac{1}{2}}}}dx\leq \int_t^1 \frac{1}{x^{\gamma}}dx\]
for all $0<t\leq 1$.

This implies that
\[\sup_{0<t\leq 1}\left \{t^{\gamma-1}\sum_{\alpha\in \mathrm{Irr}(\g)}\frac{n_{\alpha}^2}{l(\alpha)^{\frac{1}{2}}e^{tl(\alpha)^{\frac{1}{2}}}}\right\}=:C_1<\infty,\]
so that we can inductively see that
\[\sup_{0<t\leq 1}\left \{t^{\gamma-(m-1)}\sum_{\alpha\in \mathrm{Irr}(\g)}\frac{n_{\alpha}^2}{l(\alpha)^{\frac{m-1}{2}}e^{tl(\alpha)^{\frac{1}{2}}}}\right\}=:C_{m-1}<\infty.\]

Then there exist $D_1,D_2>0$ such that
\[\sum_{\alpha\in \mathrm{Irr}(\g)}\frac{n_{\alpha}^2}{l(\alpha)^{\frac{m}{2}}e^{tl(\alpha)^{\frac{1}{2}}}}\leq D_1t^{m-\gamma}+D_2~\mathrm{for~all~}0<t\leq 1\]
via the same way.

Lastly, taking the limit $t\rightarrow 0$ completes this proof.
\end{proof}

\begin{theorem}\label{comp.char.}
Let $1<p\leq 2$.
\begin{enumerate}
\item Let $G$ be a compact Lie group with dimension $n$. Then 
\[(\sum_{\pi\in \widehat{G}}\frac{1}{(1+\kappa_{\pi})^{\frac{s}{2}}}n_{\pi}||\widehat{f}(\pi)||_{S^p_{n_{\pi}}}^p)^{\frac{1}{p}}\lesssim ||f||_{L^p(G)}\]
holds if and only if $s\geq \displaystyle n(2-p)$.
\item Let $G$ be a finitely generated discrete group with polynomial growth rate $k_0$. Then 
\[(\sum_{g\in G}\frac{1}{(1+|g|)^s}|f(g)|^p)^{\frac{1}{p}}\lesssim ||\lambda(f)||_{L^p(VN(G))}\]
holds if and only if $s\geq \displaystyle k_0(2-p)$.

\item Let $\g$ be $O_2^+$ or $S_{4}^+$. Then 
\[(\sum_{k\geq 0}\frac{1}{(1+k)^s}n_k||\widehat{f}(k)||_{S^p_{k+1}}^p)^{\frac{1}{p}}\lesssim ||f||_{L^p(\g)}\]
holds if and only if $s\displaystyle  \geq 3(2-p)$.

\item Let $\g$ be $O_N^+$ or $S_{N+2}^+$ with $N\geq 3$. Then 
\[(\sum_{k\geq 0}\frac{1}{r_0^{(2-p)k}(1+k)^s}n_k||\widehat{f}(k)||_{S^p_{n_k}}^p)^{\frac{1}{p}}\lesssim ||f||_{L^p(\g)}\]
holds if and only if $s\displaystyle  \geq 4-2p$, where $r_0=\displaystyle \frac{N+\sqrt{N^2-4}}{2}$.
\end{enumerate}
\end{theorem}

\begin{proof}
``If'' parts are obtained from (\ref{HL-cpt.gp.2}), (\ref{HL-poly.disc.}), (\ref{HL-Free.Orth.-1}) and (\ref{HL-Free.Ortho.-2}). To prove the converse direction, firstly, define $l(\alpha)$ by (1) $(1+\kappa_{\pi})^{\frac{1}{2}}$ and (2) $1+|g|$ respectively. Then the assumed inequality
\[(\sum_{\alpha\in \mathrm{Irr}(\g)}\frac{n_{\alpha}}{l(\alpha)^{s}}||\widehat{f}(\alpha)||_{S^p_{n_{\alpha}}}^p)^{\frac{1}{p}}\lesssim ||f||_{L^p(\g)}\] 
implies the inequality (\ref{sobolev1}) for $\beta=s$. Then, by Proposition \ref{tool4} and Lemma \ref{conv.cond.}, we get that (1) $2n\leq \displaystyle \frac{2\beta}{2-p}$, (2) $2k_0\leq \displaystyle \frac{2\beta}{2-p}$  respectively.

In (3) and (4), consider $G$ as $SU(2)$ if $\g=O_N^+$ and $SO(3)$ if $\g=S_{N+2}^+$ for each cases. Also, denote by $\chi_k'$ the character corresponding to $\chi_k$ for each cases. Define $l(k):=1+k$. 

Firstly, in (3), for each $f\sim \displaystyle \sum_{k\geq 0}c_k\chi_k\in L^p(\g)$,
\begin{align*}
||\sum_{k\geq 0}(1+k)^{-\frac{s}{p}}c_k\chi_k||_{L^{p'}(G)}&\leq (\sum_{k\geq 0}\frac{1}{(1+k)^{s+p-2}}|c_k|^p)^{\frac{1}{p}}\\
&\lesssim ||f||_{L^p(\g)}=||f'\sim \sum_{k\geq 0}c_k\chi_k'||_{L^p(G)}.
\end{align*}
by Proposition \ref{monoidal} and Hausdorff-Young inequality.
On the other hand, in (4), for each $f\sim \displaystyle \sum_{k\geq 0}c_k\chi_k\in L^p(\g)$,
\begin{align*}
||\sum_{k\geq 0}(1+k)^{-\frac{s+2-p}{p}}c_k\chi_k||_{L^{p'}(G)}&\leq (\sum_{k\geq 0}\frac{1}{(1+k)^s}|c_k|^p)^{\frac{1}{p}}\\
&\lesssim ||f||_{L^p(\g)}=||f'\sim \sum_{k\geq 0}c_k\chi_k'||_{L^p(G)}
\end{align*}
by the similar way.

Now we can apply Proposition 6.5 and Lemma 6.2 for compact Lie groups again, so that (3) $s\geq 6-3p$ and (4) $s-p+2\geq 6-3p(\Leftrightarrow s\geq 4-2p)$ respectively.
\end{proof}

\section{Some remarks about Sidon sets, Sobolev embedding theorem and quantum torus}

As by-products of this study, we refer to an interesting lacunarity result for compact quantum groups, and present a sobolev embedding theorem type interpretation for $C(G)$ with compact Lie group and for $C_r^*(G)$ with polynomially growing group. Also, we show an explicit inequality on quantum torus $\tor^d_{\theta}$.

\subsection{Sidon set on compact quantum groups}

The study of Lacunarity, especially on Sidon sets, is one of the major subject in harmonic analysis, and recently the notion has been extended to the setting of compact quantum groups \cite{Wa16}. 
\begin{definition}
Let $\g$ be a compact quantum group.
\begin{enumerate}
\item A subset $E\subseteq \mathrm{Irr}(\g)$ is called a Sidon set if there exists $K>0$ such that 
\[||\widehat{f}||_{\ell^1(\widehat{\g})}\leq K ||f||_{L^{\infty}(\g)}~\mathrm{for~all~}f\in Pol_E(\g),\]
where $Pol_E(\g):=\left \{f\in Pol(\g):\widehat{f}(\alpha)=0~\mathrm{for~all~}\alpha\notin E\right\}$.
\item A subset $E\subseteq \mathrm{Irr}(\g)$ is called a central Sidon set if there exists $K>0$ such that 
\[||\widehat{f}||_{\ell^1(\widehat{\g})}\leq K ||f||_{L^{\infty}(\g)}~\mathrm{for~all~}f=\sum_{\alpha\in \mathrm{Irr}(\g)}c_{\alpha}\chi_{\alpha}\in Pol_E(\g).\]
\end{enumerate}
\end{definition}

Let $\g=(A,\Delta)$ be of Kac type and $E\subseteq \mathrm{Irr}(\g)$ be a central sidon set. Then [Proposition 6.4, \cite{Wa16}] implies that there exists $\mu\in M(\g)=C_r(\g)^*$ such that $\widehat{\mu}(\alpha)=(\mu((u^{\alpha}_{j,i})^*))_{1\leq i,j\leq n_{\alpha}}=Id_{n_{\alpha}}$ for all $\alpha\in E$. Since $Pol(\g)$ is dense in $C_r(\g)$, Proposition \ref{prop.rd.} still holds for $\mu\in M(\g)$.

Now, if $\g$ satisfies the assumptions of Proposition \ref{prop.rd.} and if $E\subseteq \mathrm{Irr}(\g)$ is a central sidon set, then we get
\begin{align*}
\infty & >\sup_{k\geq 0}\frac{(\sum_{\alpha\in E_k}n_{\alpha}^2)^{\frac{1}{2}}}{(1+k)^{\beta}}\\
&\geq \sup_{k\geq 0}\frac{|E_k|^{\frac{1}{2}}\min_{\alpha\in E_k}n_{\alpha}}{(1+k)^{\beta}},
\end{align*}
where $E_k:=\left \{\alpha\in E:|\alpha|=k\right\}$.

Therefore, it can not happen simultaneously that $|E|=\infty$ and $n_{\alpha}>r^{|\alpha|}~\forall \alpha\in Irr(\g)$ with $r>1$.

\begin{remark}
\begin{enumerate}
\item The above argument shows that there is no an infinite (central) Sidon set in $U^+_N$ with $N\geq 3$, which are not explained in \cite{Wa16}.
\item Shortly after this research, the author of \cite{Wa16} personally informed me another simple idea to explain $U_N^+$ cases. Under the identification $\mathrm{Irr}(U_N^+$)$\cong \mathbb{F}_2^+$, the fact that 
\[||\chi_{\alpha}||_4=(1+|\alpha|)^{\frac{1}{4}}~\forall \alpha\in \mathbb{F}_2^+\]
implies that there is no an infinite $\Lambda(4)$ set, so that there is no an infinite Sidon set on $U_N^+$ with $N\geq 2$.
\end{enumerate}
\end{remark}

\subsection{Sobolev embedding properties}\label{sobolev}

The contents of Section \ref{Sharp} can be interpreted in terms of Sobolev embeddings properties by  [Theorem 1.1, \cite{Xi16}].

For $C(G)$ with compact Lie group whose real dimension is $n$, the computations in Section \ref{Sharp} says that
\[||(1-\Delta)^{-\frac{\beta}{2}}(f)\sim \sum_{\pi\in \widehat{G}}\frac{n_{\pi}}{(1+\kappa_{\pi})^{\frac{\beta}{2}}}\mathrm{tr}(\widehat{f}(\pi)\pi)||_{L^{p'}(G)}\lesssim ||f||_{L^p(G)}\]
if and only if $\beta \displaystyle \geq \frac{n(2-p)}{p}$ for each $1<p\leq 2$. Moreover, it is equivalent to that
\[||(1-\Delta)^{-\frac{\beta }{2}(\frac{1}{p}-\frac{1}{q})}(f)||_{L^q(G)}\lesssim ||f||_{L^p(G)}\]
if and only if $\beta\geq n$ for each $1<p<q<\infty$. If we define the space $H^s_p(G):=\left \{f\in L^p(G):(1-\Delta)^{\frac{s}{2}}(f)\in L^p(G)\right\}$ as an analogue of Bessel potential space, then the above result is interpreted as
\begin{equation}\label{Sob1}
H_p^{s}(G)\subseteq L^q(G)~\mathrm{if~and~only~if~}s\geq n(\frac{1}{p}-\frac{1}{q})
\end{equation}
for each $1<p<q<\infty$.

On the other hand, if $G$ is a finitely generated discrete group with polynomial growth rate $k_0$, then we define infinitesimal generator $L$ on $C_r^*(G)$ by $\lambda_g\mapsto -|g|\lambda_g$ for all $g\in G$. Then we can derive the Sobolev embedding property of non-commutative spaces $L^p(VN(G))$ as follows.

\begin{equation}\label{Sob2}
||(1-L)^{-\beta(\frac{1}{p}-\frac{1}{q})}(\lambda(f))||_{L^q(VN(G))}\lesssim ||\lambda(f)||_{L^p(VN(G))} \mathrm{~if~and~only~if}~\beta\geq k_0
\end{equation}
for each $1<p<q<\infty$.

The reader may consider another natural infinitesimal generator $L':\lambda_g\mapsto -|g|^2\lambda_g$, but it does not make an essential difference in replacing $(1-L)$ with $(1-L')^{\frac{1}{2}}$.

\subsection{Hardy-Littlewood inequality on Quantum torus}

Quantum torus $\tor^d_{\theta}$ is a widely studied example of ``quantum space''. In this case, we can establish Hardy-Littlewood inequality for $\tor^d_{\theta}$, which is the same form as for $\tor^d$. A proof can be given by repeating the proof of Theorem \ref{paley-CQG}.

\begin{remark}\label{q.torus}
For a quantum torus $\tor^d_{\theta}$, for each $1<p\leq 2$, we have that
\begin{equation}\label{HL-QT}
(\sum_{m\in \z^d}\frac{1}{(1+||m||_1)^{d(2-p)}}|\widehat{x}(m)|^p)^{\frac{1}{p}}\lesssim ||x||_{L_p(\tor^d_{\theta})}.
\end{equation}
\end{remark}

{\bf Acknowledgement.}  The author is grateful to Hun Hee Lee for helpful comments on this work. Also, he would like to thank Simeng Wang for discussing the Sidon set of compact quantum groups.

\end{document}